\documentclass[11pt,a4paper]{amsart}
\usepackage{amssymb}
\usepackage{amsmath}
\usepackage{pifont}
\usepackage[mathb]{mathabx}
\usepackage[usenames]{color}
\usepackage{amsfonts}
\usepackage[all]{xypic}
\usepackage{graphicx}
\usepackage{psfrag}
\usepackage{verbatim}
\usepackage{pstricks}
\usepackage[latin1]{inputenc}
\usepackage{mathrsfs}
\usepackage[all,knot]{xy}

\newtheorem{thm}{Theorem}
\newtheorem{corollary}{Corollary}
\newtheorem{lemma}{Lemma}
\newtheorem{example}{Example}
\newtheorem{definition}{Definition}
\newtheorem{remark}{Remark}
\newtheorem{prop}{Proposition}
\newtheorem{claim}{Claim}
\newtheorem{observation}{Observation}
\newtheorem*{conjecture}{Conjecture}
\newtheorem*{rlemma}{Reduction Lemma}
\newtheorem*{mprop}{Main Proposition}

\newenvironment{Remark}{\begin{remark}\rm}{\end{remark}}

\def\et{\quad\mbox{and}\quad}
\def\vel{\;\mbox{or}\;}
\def\for{\quad\mbox{for }}

\def\where{\quad\mbox{where }}

\def\N{\mathbb{N}}

\def\R{\mathbb{R}}
\def\Z{\mathbb{Z}}

\def\epsilon{\varepsilon}
\def\s{{\sigma}}
\def\S{\widetilde{\sigma}}
\def\b{{\rm{b_1}}}
\def\l{{l}}
\begin{document}

\author{Peter Feller}
\address{Boston College, Department of Mathematics, Maloney Hall, Chestnut Hill, MA 02467, United States}
\email{peter.feller.2@bc.edu}
\keywords{Positive braids, signature, topological slice genus}
\subjclass[2010]{57M25,  57M27}

\title[The signature of positive braids is linearly bounded]{The signature of positive braids is linearly bounded by their first Betti number}

\begin{abstract}
We provide linear lower bounds for the signature of positive braids in terms of the three genus of their braid closure.
This yields linear bounds for the topological slice genus of knots that arise as closures of positive braids.\\[0.3cm]
\end{abstract}

\maketitle

\renewcommand{\thesubsection}{\arabic{subsection}}

\section{Introduction}

Let $B_b$ be the {braid group} on $b$ strands as introduced by Artin~\cite{Artin_TheorieDerZoepfe}.
A group presentation for $B_b$ is given by the (standard) generators $ a_1,\cdots, a_{b-1}$ satisfying the {braid relations}
\[ a_{i} a_{j}= a_{j} a_{i} \for |i-j|\geq 2\et  a_{i} a_{j} a_{i}= a_{j} a_{i} a_{j} \for |i-j|=1. \]
A \emph{positive braid} on $b$ strands is an element $\beta$ in $B_b$ that can be written as positive braid word $ a_{s_{1}} a_{s_{2}}\cdots a_{s_{\l}}$ with $s_{i}\in\{1,\cdots,b-1\}$,
where $\l$, which is independent of the choice of a positive braid word for $\beta$,
is called the \emph{length} of $\beta$.
Links obtained as the (standard) closure of positive braids contain much studied classes of links;
for example, torus links, algebraic links and Lorenz links.
The \emph{signature} is a classical link invariant introduced by Trotter~\cite{Trotter_62_HomologywithApptoKnotTheory}.
The \emph{first Betti number} $\b(L)$ of a link $L$ is defined to be the smallest first Betti number of Seifert surfaces for $L$.
The \emph{(three) genus} of a knot is half its first Betti number.
By the first Betti number $\b(\beta)$ and the signature $\s(\beta)$ of a braid $\beta$, we mean the first Betti number and the signature, respectively, of the link obtained as the closure of $\beta$.
A braid is said to be \emph{non-trivial} if its closure is not an unlink; that is if its first Betti number is non-zero.

In this paper we relate the first Betti number and the signature for positive braids.
All that follows is motivated by the following
conjecture.
\begin{conjecture}\label{conj:c>0}For all non-trivial positive braids $\beta$, the signature is linearly bounded as follows
$\b(\beta)\geq{-\s}(\beta)>\frac{1}{2}\b(\beta).$
\end{conjecture} This would be optimal since there are families of positive braids on which the ratio $\frac{{-\s}}{\b}$ gets arbitrarily close to $\frac{1}{2}$.
For example, positive braids $\beta_n$ that have the $T(n,n+1)$ torus knots as closures.
We provide a linear lower bound as in the conjecture. However, the linear factor is smaller than $\frac{1}{2}.$
\begin{thm}
\label{thm:K>0}
For all positive braids $\beta$, the signature is linearly
bounded from below as follows
${-\s}(\beta)\geq \frac{\b(\beta)}{100}.$
\end{thm}

We put Theorem~\ref{thm:K>0} in context by recapitulating previously established facts about the first Betti number, the signature, and the slice genus of positive braids.
The first Betti number of positive braids is fully understood since Bennequin's inequality provides the following formula~\cite{Bennequin_Entrelacements}.
\begin{equation}\label{eq:b}
\b(\beta)\;\;=\;\;\l(\beta)-b+c\;\;\text{for every positive braid }\beta,
\end{equation}
where $\l(\beta)$ is the length of $\beta$, $b$ is the number of strands of $\beta$, and $c$ 
equals $1$ plus the number of generators $ a_{i}$ that are not used in a positive braid word for $\beta$. 
For the signature one has that $-\b(L)\leq\s(L)\leq\b(L)$ holds for all links $L$. This is immediate from the definition of the signature; see Section~\ref{sec:sig}.
Rudolph showed that the signature is strictly negative for non-trivial positive braids~\cite{Rudolph_83_PosBraidPosSig}.
Stoimenow provided a monotonically growing function $f\colon\N\to \R_{\geq0}$ of order $n^{\frac{1}{3}}$ such that ${-\s}(\beta)\geq f(\b(\beta))$~\cite{Stoimenow_08}.
By a result of Murasugi, $\frac{|\s|}{2}$ is a lower bound
for the smooth slice genus of knots~\cite{Murasugi_OnACertainNumericalInvariant}. 
Kauffman and Taylor generalized Murasugi's result to the topological slice genus~\cite{KauffmanTaylor_76_SignatureOfLinks},
which is notoriously difficult to determine even for torus knots.
Therefore, Rudolph and Stoimenow's results provide lower bounds for the topological slice genus of knots that are closures of positive braids.
Similarly, Theorem~\ref{thm:K>0} yields the following corollary.
\begin{corollary}
The topological slice genus of positive braid knots
is at least one percent of their genus. In other words, for a knot obtained as the closure of a positive $b$-braid $\beta$, the topological slice genus is at least
$\frac{\l(\beta)-b+1}{200}$.\qed
\end{corollary}
In the stronger, smooth setting much more is known.
Indeed, on knots that are closures of positive braids the smooth slice genus 
even agrees with the genus. 
This follows from Rudolph's slice-Bennequin inequality~\cite{rudolph_QPasobstruction}, which is based on the local Thom conjecture as proved by Kronheimer and Mrowka~\cite{KronheimerMrowka_Gaugetheoryforemb}.

We prove Theorem~\ref{thm:K>0} via the study of the asymptotic signature.
Gambaudo and Ghys observed that on the braid group on $b$ strands the signature is a quasi-morphism of defect $b-1$:
for any two $b$-braids $\alpha$ and $\beta$, we have \[|\s(\alpha\beta)-\s(\alpha)-\s(\beta)|\leq b-1;\]
see \cite[Proposition 5.1]{GambaudoGhys_CommutatorsAndDiffeosOfSurfaces}.
Therefore, the homogenization \[\S(\beta)=\lim_{i\to\infty}\frac{\s(\beta^i)}{i},\]
called the \emph{asymptotic signature} of $\beta$, is well-defined.
Noting that
\[\lim_{i\to\infty}\frac{\b(\beta^i)}{i}=\lim_{i\to\infty}\frac{l(\beta^i)}{i}=l(\beta),\]
we see that the above conjecture 
implies the following homogenized analog.
\begin{conjecture}\label{conjHomo:c>0}For all positive braids $\beta$, the homogenization of the signature is linearly bounded as follows
$\l(\beta)\geq{-\S}(\beta)\geq \frac{1}{2}\l(\beta).$
\end{conjecture}
We provide such a linear bound, but our factor is $\frac{1}{16}$ rather than $\frac{1}{2}$.
\begin{thm}\label{thm:1/16}
  For every positive braid $\beta$, we have
$-\S(\beta)\geq\frac{1}{16}\l(\beta).$
\end{thm}

Next we recall known linear bounds for the signature of non-trivial positive braids with fixed number of strands $b$,
which of course yield analogs for the asymptotic signature.
If $b=2$, then the closure of $\beta$ is the torus link $T(2,\l(\beta))$ and ${-\s}(\beta)=\b(\beta)=\l(\beta)-1$.
For $b=3$, Stoimenow (and by different methods Yoshiaki Uchida) has shown ${-\s}(\beta)> \frac{1}{2}\b(\beta)$ \cite{Stoimenow_08}.
The case $b=4$ is our main concern in this text. 
\begin{mprop}
For every positive 4-braid $\beta$, we have
$-\S(\beta)\geq\frac{5}{12}\l(\beta).$
\end{mprop}

Our interest in the Main Proposition stems from the fact that it implies Theorem~\ref{thm:K>0} and Theorem~\ref{thm:1/16}.
Indeed, Theorem~\ref{thm:1/16} follows immediately from the Main Proposition and the following observation, which we prove at the end of Section~\ref{sec:sig}.
\begin{rlemma}
 Let $C$ be a positive constant such that ${-\s}(\beta)> C\b(\beta)$ (respectively $-\S(\beta)\geq Cl(\beta)$) holds for all non-trivial positive $b$-braids. Then
\[-\s(\beta)>\tilde{C}\b(\beta) \;\;\left(\text{respectively }-\S(\beta)\geq \tilde{C}l(\beta)\right),\where \tilde{C}=\tfrac{C(b-1)-1}{b},\]
holds for all non-trivial positive braids $\beta$.
\end{rlemma}
\begin{Remark}
 The Reduction Lemma remains true when the strict inequalities are replaced by inequalities.
\end{Remark}
The upshot of the Main Proposition is that it provides a bound with a factor that is strictly larger than $\frac{1}{3}$.
If one is only interested in the fact that some linear bound exists for positive $4$-braids,
Stoimenow provided such a bound with the factor $\frac{2}{11}$~\cite[Theorem~4.2]{Stoimenow_08}.
In fact, we even establish
\begin{equation}\label{eq:thirdsigbound}
 {-\s}(\beta)>\frac{1}{3}\b(\beta)\;\;\text{ for all non-trivial positive $4$-braids $\beta$,}
\end{equation}
without relying on the Main Proposition; see Proposition~\ref{prop:thirdsigbound}.
However, these results do not provide linear bounds for general positive braids when combined with the Reduction Lemma.
For example, applying the Reduction Lemma to \eqref{eq:thirdsigbound} just recovers Rudolph's result that
for all non-trivial positive braid $\beta$, one has ${\s}(\beta)<0$.

In Section~\ref{sec:Sigtosig} we use the Main Proposition to prove Theorem~\ref{thm:K>0}. Sections~\ref{sec:SurfacesAndFenceDiagrams} and~\ref{sec:sig} contain generalities on positive braids and their signature,
which are applied to $4$-braids in Section~\ref{sec:4braids} to prove \eqref{eq:thirdsigbound} and the Main Proposition.
All considerations will be restricted to braids with $c=1$ because the signature and the first Betti number are additive on disjoint unions of positive braids. 

We conclude the introduction with some evidence to support our conjectured $\frac{1}{2}\b$-bound. 
For positive $3$-braids, the conjecture holds.
And, for positive $4$-braids up to $17$ crossings, it is also checked; compare Stoimenow's table in~\cite{Stoimenow_08}.
All non-trivial torus links satisfy ${-\s}>\frac{1}{2}\b$, which can be checked using the Gordon-Litherland-Murasugi reduction formulas~\cite[Theorem~5.2]{GLM}.
More generally, the conjecture holds for all algebraic links.
This can be checked using the formula provided by Shinohara; see~\cite{Shinohara_71_OnTheSignatureOfKnotsAndLinks}.
Using Shinohara's formula one can also check 
that the conjecture holds for a lot of other families of positive braid knots that are cables of positive braid knots.
An improved version of the first inequality of~\cite[Theorem~3]{Baader_13_theVolumeOfPositiveBraidLinks}
shows that the conjecture holds for positive braids
that are given by a positive braid word of the form $a_{s_1}^{k_1}\cdots a_{s_r}^{k_r}$ with $k_i\geq2$.

\section*{Acknowledgements} I thank Sebastian Baader for all the fruitful discussions and for providing Figure~\ref{fig:fencediag}.
Thanks also to Livio Liechti for valuable suggestions. Finally, I thank the referee for many suggestions and improvements to the text.
This project is supported by the Swiss National Science Foundation.

\section{From asymptotic signature to signature}\label{sec:Sigtosig}
In this section we provide consequences of the Main Proposition, including a proof of Theorem~\ref{thm:K>0}.
In addition to the Main Proposition, this uses the Reduction Lemma and Proposition~\ref{prop:thirdsigbound},
which are proven in Section~\ref{sec:sig} and Section~\ref{sec:4braids}.

Having a linear bound ${-\s}\geq C\b$ for all positive braids on $b$ or fewer strands
yields a bound $-\S\geq C\l$ for all positive braids on $b$ or fewer strands.
The converse is true up to an additive constant:

\begin{lemma}\label{lemma:asymsig>c->sig>cl-(b-1)}
Let $C$ be a positive constant such that
\[{-\S}(\beta) = \lim_{i\to\infty}\frac{{-\s}(\beta^i)}{i} \geq  C\l(\beta)\]
for all positive $b$-braids $\beta$. Then, for every positive $b$-braid $\beta$, we have
\[{-\s}(\beta) \geq  C\l(\beta)-b+1.\]
\end{lemma}

Lemma~\ref{lemma:asymsig>c->sig>cl-(b-1)} is an immediate consequence of the fact that the signature and its homogenization stay close.
That is, for every $\beta$ in $B_b$, we have
\[|\s(\beta)-\S(\beta)| \leq  b-1,\]
which follows from $\s$ being  a quasi-morphism of defect $b-1$.

For positive $4$-braids with non-split closure (i.e.\ $c=1$), applying Lemma~\ref{lemma:asymsig>c->sig>cl-(b-1)} to the Main Proposition and using \eqref{eq:b} yields
${-\s}(\beta)\geq 
\frac{5}{12}(\b(\beta)+3)-3$. Therefore, we get the following affine signature bound for positive $4$-braids.

\begin{corollary}\label{cor:5/12} If $\beta$ is a positive $4$-braid, then ${-\s}(\beta)\geq\frac{5}{12}\b(\beta)-\frac{7}{4}$.\qed
\end{corollary}

Corollary~\ref{cor:5/12} can be used to prove the following Proposition.

\begin{prop}\label{prop:siggeqCb}
If $\beta$ is a positive $4$-braid, then \[{-\s}(\beta)\geq (\frac{1}{3}+\frac{1}{75})\b(\beta).\]
\end{prop}

In turn, Proposition~\ref{prop:siggeqCb} implies Theorem~\ref{thm:K>0} by the Reduction Lemma
since we have \[\tilde{C}=\frac{C(4-1)-1}{4}=\frac{1}{100}\for C=\frac{1}{3}+\frac{1}{75}.\]

\begin{proof}[Proof of Proposition~\ref{prop:siggeqCb}]
Set $C=\frac{1}{3}+\frac{1}{75}$. 
Corollary~\ref{cor:5/12} can be written as
\begin{equation}\label{eq:geqC+5/12-C}
{-\s}(\beta)\geq\frac{5}{12}\b(\beta)-\frac{7}{4} = C\b(\beta)+(\frac{5}{12}-C)\b(\beta)-\frac{7}{4}
\end{equation}
for all positive $4$-braids $\beta$.
The constant $C$ is chosen such that \eqref{eq:geqC+5/12-C} yields
 ${-\s}(\beta)\geq C\b(\beta)$ whenever $\b(\beta)\geq 25$.
On the other hand, we have ${-\s}(\beta)> \frac{\b(\beta)}{3}$ for all non-trivial positive $4$-braids; see Proposition~\ref{prop:thirdsigbound},
which can be written as ${-\s}(\beta)\geq \frac{\b(\beta)}{3}+\frac{1}{3}$. 
In particular,
\[
{-\s}(\beta) \geq  
\frac{\b(\beta)}{3}+\frac{1}{75}\b(\beta)
\]
for all positive $4$-braids with $\b\leq 25$.
\end{proof}

If one were able to strengthen Corollary~\ref{cor:5/12} to a linear bound for the signature with factor $\frac{5}{12}$ or even $\frac{1}{2}$,
then Theorem~\ref{thm:K>0} would follow immediately from the Reduction Lemma with factor $\frac{1}{16}$ or $\frac{1}{8}$, respectively, rather than $\frac{1}{100}$.

Applying Lemma~\ref{lemma:asymsig>c->sig>cl-(b-1)} to Theorem~\ref{thm:1/16} yields
the following affine linear bound.

\begin{corollary}\label{cor:1/16} 
For every positive braid $\beta$ on at most $b$ strands, we have \[\pushQED{\qed}{-\s}(\beta) \geq \frac{1}{16}\b(\beta)-\frac{15}{16}(b-1).\qedhere\popQED\]
\end{corollary}

Corollary~\ref{cor:1/16} can also be proved by applying a slight modification of the proof of the Reduction Lemma to Corollary~\ref{cor:5/12}.
With this, the constant $\frac{15}{16}(b-1)$ is improved to be $\frac{7}{4}\lfloor\frac{b}{4}\rfloor$. 
Corollary~\ref{cor:1/16} shows that 
the topological slice genus of positive braid knots 
grows asymptotically at least as fast as $\frac{1}{16}$ of the genus.

\section{Minimal Seifert surfaces and fence diagrams}\label{sec:SurfacesAndFenceDiagrams}
In this section we recall how to switch between three ways of viewing positive braids;
namely, by their braid diagrams, by their associated minimal Seifert surface,
and by their fence diagrams.
In this article, a \emph{Seifert surface} for a link $L$ in $\R^3$
is a possibly non-connected, oriented, embedded surface with oriented boundary $L$.
Seifert surfaces that minimize the first Betti number are called {minimal Seifert surfaces}.

When drawing braid diagrams for a braid $ a_{i_1} a_{i_2}\cdots a_{i_l}$, we start with the leftmost generator, then draw the second generator on top, and so one.
For example, the positive 3-braid $ a_1 a_2$ is represented by $~\xygraph{
!{0;/r0.7pc/:}
[u(2.5)]
[d(1.4)]!{\vcrossneg}
[ul]!{\xcapv[1]@(0)} [ld]
[ur]!{\vcrossneg}
[urr]!{\xcapv[1]@(0)} [ld]
}\hspace{-0.5pc}$, where the generators $a_1$ and $a_2$ are represented by $~\xygraph{
!{0;/r0.7pc/:}
[u(2)]
[d(1.5)l]!{\vcrossneg}
[urr]!{\xcapv[1]@(0)} [ld]
}\hspace{-0.5pc}$
and
$~\xygraph{
!{0;/r0.7pc/:}
[u(2)]
[d(1.5)]!{\vcrossneg}
[ul]!{\xcapv[1]@(0)} [ld]
}~$, respectively.
We orient all strands of braids in the same direction, say upwards, and closures of braids are oriented accordingly.

Closures of positive braids have a unique minimal Seifert surface.
Indeed, non-split (i.e.\ $c=1$) positive braids links are fibered~\cite{Stallings_78_ConstructionsOfFibredKnotsandLinks} and, therefore, the fiber surface is the unique minimal Seifert surface.
If a positive braid link has $c>1$,
then it also has a unique minimal Seifert surface,
which is given as the disjoint union of minimal Seifert surfaces of the $c$ pieces.
We note that for a positive braid $\beta$,
$c$ is equal to the number of components of the closure of $\beta$ that can be separated by spheres.
This follows from the fact that the closure of a positive braid word in which all generators appear are fibered and, therefore, cannot be separated by a sphere.

Since in a braid diagram associated with a positive braid word all crossings are positive,
replacing every crossing with a horizontal line still allows one to describe positive braids.
This yields \emph{fence diagrams} of positive braids as depicted in Figure~\ref{fig:fencediag}; compare Rudolph~\cite{Rudolph_98_QuasipositivePlumbing}.
A nice feature of fence diagrams is that the fence diagram of a positive braid $\beta$,
seen as a graph in $\R^3$,
is a deformation retract of the minimal Seifert surface of the closure of $\beta$.
\begin{figure}[h]
\begin{center}
\psfrag{=}{$~$}
\includegraphics[width=.4\textwidth]{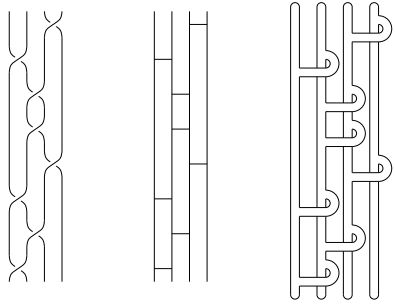}
\caption{The braid diagram, the fence diagram, and the unique associated Seifert surface for the positive 4-braid word $ a_1 a_2 a_1 a_3 a_2 a_2 a_1 a_3 $.}
\label{fig:fencediag}
\end{center}
\end{figure}
In particular, the first Betti number of a positive braid is equal to the first Betti number of the corresponding fence diagram.
We feel that all the above becomes very clear by considering an example.
Figure~\ref{fig:fencediag} provides the braid diagram, 
the fence diagram, and the minimal Seifert surface of the closure, for one positive 4-braid word.
In what follows we depict positive braids using fence diagrams.

As there are several positive braid words for most positive braids, there are also several fence diagrams.
Two fence diagrams for the same positive braid are related by moves corresponding to the braid relations.
The relation $ a_i a_j= a_j a_i$ for $|i-j|\geq 2$ is incorporated by looking at fence diagrams up to planar isotopy, e.g.\
$\xygraph{
!{0;/r0.7pc/:}
[u(0.75)]!{\xcapv[1.5]@(0)}
[lu]!{\xcapv[1.5]@(0)}
[lu]!{\xcapv[1.5]@(0)}
[lu]!{\xcapv[1.5]@(0)}
[u(0.5)rr] !{\xcaph[1]@(0)}
[d(0.5)lll] !{\xcaph[1]@(0)}
}\hspace{-0.75pc}=~\xygraph{
!{0;/r0.7pc/:}
[u(0.75)]!{\xcapv[1.5]@(0)}
[lu]!{\xcapv[1.5]@(0)}
[lu]!{\xcapv[1.5]@(0)}
[lu]!{\xcapv[1.5]@(0)}
[u(0.5)] !{\xcaph[1]@(0)}
[d(0.5)r] !{\xcaph[1]@(0)}
}\hspace{-0.75pc}=~\xygraph{
!{0;/r0.7pc/:}
[u(0.5)]!{\xcapv[1]@(0)}
[lu]!{\xcapv[1]@(0)}
[lu]!{\xcapv[1]@(0)}
[lu]!{\xcapv[1]@(0)}
[u(0.5)]!{\xcaph[1]@(0)}
[r] !{\xcaph[1]@(0)}}
\hspace{-0.25pc}$.
The braid relation $ a_i a_j a_i= a_j a_i a_j$ for $|i-j|=1$ corresponds to the move
$~\xygraph{
!{0;/r0.7pc/:}
[u(1)]!{\xcapv[2]@(0)}
[lu]!{\xcapv[2]@(0)}
[lu]!{\xcapv[2]@(0)}
[u(0.5)r] !{\xcaph[1]@(0)}
[d(0.5)ll] !{\xcaph[1]@(0)}
[d(0.5)] !{\xcaph[1]@(0)}
}\hspace{-1pc}=~\xygraph{
!{0;/r0.7pc/:}
[u(1)]!{\xcapv[2]@(0)}
[lu]!{\xcapv[2]@(0)}
[lu]!{\xcapv[2]@(0)}
[u(0.5)] !{\xcaph[1]@(0)}
[d(0.5)] !{\xcaph[1]@(0)}
[d(0.5)ll] !{\xcaph[1]@(0)}
}\hspace{-1pc}.$

\section{Signature of links and braids}\label{sec:sig} In this section we recall the definition and properties of the signature. Also, we prove the Reduction Lemma.

The signature of a link $L$ in $\R^3$, denoted by $\s(L)$, is defined as follows.
Choose any Seifert surface $F$ for $L$ and define the \emph{Seifert form}\textemdash
a bilinear form $S\colon H_1(F)\times H_1(F)\to \Z$ on the first integer homology group of $F$\textemdash
as follows: represent classes $a,b$ in $H_1(F)$ by immersed curves $\gamma,\delta$ in $F$
and let $\delta^+$ be the curve obtained by moving $\delta$ a small amount along the outward pointing normal vector field of $F$.
Then $S(a,b)$ is set to be the linking number of $\gamma$ and $\delta^+$,
where the linking number of non-embedded disjoint curves in $\R^3$ is defined
by taking embedded perturbations contained in disjoint neighborhoods of the original curves.
Writing $S$ in a basis for $H_1(F)$ yields a \emph{Seifert matrix} $A$ for $L$.
The number of positive eigenvalues minus the number of negative eigenvalues of the symmetrization of $A$ 
yields a link invariant, the \emph{signature}; see~\cite{Trotter_62_HomologywithApptoKnotTheory} and~\cite{Murasugi_OnACertainNumericalInvariant}.
There are different sign conventions for the signature in the literature.
We choose sign conventions (for example in the definition of the linking number) such that positive braids have negative signature. In particular,
$\s(a_1^n)=-n+1$ (rather than $n-1$) for all positive integers $n$.

The signature is additive on disjoint and connected sums of links.
Note that the connected sum of links with more than one component is only well-defined if one specifies which components to connect.
However, all possible resulting links have the same signature.

\begin{lemma}\label{lemma:sigandunion}
If a link $L$ is a disjoint or a connected sum of links $L_1, \cdots, L_k$, then $\s(L)=\sum_{i=1}^{k}\s(L_i)$.
\end{lemma}

Lemma~\ref{lemma:sigandunion} follows from the fact that the direct sum of Seifert matrices $A_i$ for $L_i$ provides a Seifert matrix $A=\oplus_{i=1}^k A_i$ for $L$.

The fact that deleting the first row and the first column of a symmetric matrix changes its signature by at most $\pm1$ yields the following lemma; see~\cite{tristram}.
\begin{lemma}\label{lemma:bandcutting}
 If a Seifert surface $F$ for a link $L$ is obtained from a Seifert surface $F'$ for a link $L'$ by adding or deleting a 1-handle, then \[\pushQED{\qed} \s(L')-1\leq\s(L)\leq\s(L')+1.\qedhere\popQED\]
\end{lemma}
 Let us now discuss properties of the signature of braids.
Firstly, if we permute braids cyclically, then they have the same closure and thus the same signature:
let $\beta= a_{i_1}^{\epsilon_{1}} a_{i_2}^{\epsilon_{2}}\cdots a_{i_l}^{\epsilon_{l}}$ be a braid in some $B_b$, then
$ a_{i_2}^{\epsilon_{2}}\cdots a_{i_l}^{\epsilon_{l}} a_{i_1}^{\epsilon_{1}}$ has the same closure and thus the same signature as $\beta$.
If we add or delete a generator in a braid word, then the signature of the corresponding braid changes by at most $\pm1$. This is a consequence of Lemma~\ref{lemma:bandcutting}.
As mentioned in the introduction, $\s$ is a quasi-morphism on the $b$-strand braid group. 

\begin{lemma}\label{lemma:SigIsaQM}
 For any two $b$-braids $\alpha,\beta$, we have $|\s(\alpha\beta)-\s(\alpha)-\s(\beta)|\leq b-1$.
\end{lemma}

Using the fact that cyclic permutations of a braid have the same signature, we can state this as follows.

\begin{corollary}\label{cor:SigIsaQM}
 For $b$-braids $\alpha,\beta,\gamma$, we have \[\pushQED{\qed}|\s(\alpha\gamma\beta)-\s(\alpha\beta)-\s(\gamma)|\leq b-1.\qedhere\popQED\]
\end{corollary}

Lemma~\ref{lemma:SigIsaQM} is based on Lemma~\ref{lemma:bandcutting}
and the fact that there is a Seifert surface for $\alpha\beta$ that can be obtained from the connected sum of Seifert surfaces for $\alpha$ and $\beta$ by adding $b-1$ 1-handles;
see Gambaudo and Ghys for a proof~\cite[Proof of Proposition 5.1]{GambaudoGhys_CommutatorsAndDiffeosOfSurfaces}.

\begin{Remark}\label{rem:SigIsaQM}
 If $\alpha$ or $\beta$ can be written as a braid without one or several generators $ a_i$,
 then the statement of Lemma~\ref{lemma:SigIsaQM} is true with defect strictly smaller than $b-1$. Namely,
\[|\s(\alpha\beta)-\s(\alpha)-\s(\beta)|\leq b-\max\{c(\alpha),c(\beta)\},\]
where $c(\alpha)$ and $c(\beta)$ denote 1 plus the number of generators that are not needed in a braid word for $\alpha$ and $\beta$, respectively.
In particular, if $\alpha,\beta\in B_b$, then \[\s(\alpha\beta)+\s( a_i^n)-1\leq\s(\alpha a_i^n\beta)\leq\s(\alpha\beta)+\s( a_i^n)+1,\]
for all integers $n$ and all generators $a_i$.
\end{Remark}

Before using the above on positive $4$-braids to prove the Main Proposition, 
we prove the Reduction Lemma.

\begin{proof}[Proof of the Reduction Lemma]

We prove the statement only for $\s$; the proof for $\S$ is similar. Let $\beta$ be a non-trivial positive braid word in some braid group $B_n$.
Without loss of generality, we assume that every generator $a_{i}$ with $1\leq i\leq n-1$ is contained in $\beta$ at least once (i.e.~$c=1$).

The idea of the proof is to delete generators in $\beta$ such that a connected sum of braids on $b$ or fewer strands remains.

 For $i$ in $\{1,2,\ldots,b\}$, we denote by $\beta(i)$ the braid obtained from $\beta$ by deleting all but one (say the leftmost) $a_k$ for all $k$ in $\{i,i+b,i+2b,i+3b,\ldots \}$.
Figure~\ref{fig:betaifrombeta} illustrates how $\beta(i)$ is obtained from $\beta$.
\begin{figure}[h]
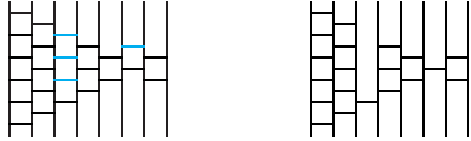

\centering
$~\xygraph{
!{0;/r0.7pc/:}
[u(3)]!{\xcapv[6]@(0)}
[lu]!{\xcapv[6]@(0)}
[lu]!{\xcapv[6]@(0)}
[lu]!{\xcapv[6]@(0)}
[lu]!{\xcapv[6]@(0)}
[lu]!{\xcapv[6]@(0)}
[lu]!{\xcapv[6]@(0)}
[lu]!{\xcapv[6]@(0)}
[u(0.5)] !{\xcaph[1]@(0)}
[d(0.5)] !{\xcaph[1]@(0)}
[d(0.5)] !{\color{cyan}\xcaph[1]@(0)}
[d(0.5)] !{\color{black}\xcaph[1]@(0)}
[d(0.5)] !{\xcaph[1]@(0)}
[d(0.5)] !{\xcaph[1]@(0)}
[u(1)l] !{\color{cyan}\xcaph[1]@(0)}
[d(0.5)] !{\color{black}\xcaph[1]@(0)}
[dl] !{\xcaph[1]@(0)}
[u(2)lllllll] !{\xcaph[1]@(0)}
[d(0.5)] !{\xcaph[1]@(0)}
[d(0.5)] !{\color{cyan}\xcaph[1]@(0)}
[d(0.5)] !{\color{black}\xcaph[1]@(0)}
[d(0.5)] !{\xcaph[1]@(0)}
[u(1)lllll] !{\xcaph[1]@(0)}
[d(0.5)] !{\xcaph[1]@(0)}
[d(0.5)] !{\color{cyan}\xcaph[1]@(0)}
[d(0.5)] !{\color{black}\xcaph[1]@(0)}
[u(0.5)llll] !{\xcaph[1]@(0)}
[d(0.5)] !{\xcaph[1]@(0)}
[d(0.5)] !{\xcaph[1]@(0)}
[lll] !{\xcaph[1]@(0)}
[d(0.5)] !{\xcaph[1]@(0)}
[d(0.5)ll] !{\color{black}\xcaph[1]@(0)}
}
\hspace{10cm}
~\xygraph{
!{0;/r0.7pc/:}
[u(3)]!{\xcapv[6]@(0)}
[lu]!{\xcapv[6]@(0)}
[lu]!{\xcapv[6]@(0)}
[lu]!{\xcapv[6]@(0)}
[lu]!{\xcapv[6]@(0)}
[lu]!{\xcapv[6]@(0)}
[lu]!{\xcapv[6]@(0)}
[lu]!{\xcapv[6]@(0)}
[u(0.5)] !{\xcaph[1]@(0)}
[d(0.5)] !{\xcaph[1]@(0)}
[dr] !{\xcaph[1]@(0)}
[d(0.5)] !{\xcaph[1]@(0)}
[d(0.5)] !{\xcaph[1]@(0)}
[u(0.5)] !{\xcaph[1]@(0)}
[dl] !{\xcaph[1]@(0)}
[u(2)lllllll] !{\xcaph[1]@(0)}
[d(0.5)] !{\xcaph[1]@(0)}
[dr] !{\xcaph[1]@(0)}
[d(0.5)] !{\xcaph[1]@(0)}
[u(1)lllll] !{\xcaph[1]@(0)}
[d(0.5)] !{\xcaph[1]@(0)}
[dr] !{\xcaph[1]@(0)}
[u(0.5)llll] !{\xcaph[1]@(0)}
[d(0.5)] !{\xcaph[1]@(0)}
[d(0.5)] !{\xcaph[1]@(0)}
[lll] !{\xcaph[1]@(0)}
[d(0.5)] !{\xcaph[1]@(0)}
[d(0.5)ll] !{\color{black}\xcaph[1]@(0)}
}\hspace{-3.8pc}$

\caption{A fence diagram of a positive braid $\beta$ on $8$ strands (left) with indications ({blue}) what generators to delete to obtain $\beta(3)$ (right) if $b$ equals $3$.
The closure of $\beta(3)$ is a connected sum of the closures of two $3$-braids and a $2$-braid.}
\label{fig:betaifrombeta}
\end{figure}
The closure of such a $\beta(i)$ is a connected sum of closures of positive braids on $b$ or fewer strands.
Since we have $\b(\beta)=\sum_{k=1}^{n-1}(\sharp\{ a_k\text{ in }\beta\}-1)$, there is an $i$ such that \begin{equation}\label{eq:bettibetai}
\b(\beta(i))\geq\frac{b-1}{b}\b(\beta).\end{equation} We fix such an $i$. 
Let $B_1,\dots, B_l$ be positive braids on at most $b$ strands such that the closure of $\beta(i)$ is the connected sum of the closures of the $B_j$. 
Thus, additivity of the first Betti number and the signature on connected sums (see Lemma~\ref{lemma:sigandunion}),
the assumption ${-\s}>C\b$ for non-trivial positive braids on $b$ strands,
and~\eqref{eq:bettibetai} yield
\begin{align*}
{-\s}(\beta(i))
= \sum_{j=1}^{l}{-\s}(B_j)
>\sum_{j=1}^{l}C\b(B_j)
= C\b(\beta(i))
\geq \frac{C(b-1)}{b}\b(\beta).
\end{align*}
The braid $\beta(i)$ is obtained from $\beta$ by deleting $\b(\beta)-\b(\beta(i))\leq\frac{1}{b}\b(\beta)$ of the generators.
Since by Lemma~\ref{lemma:bandcutting} deleting one generator changes the signature by at most $\pm1$, we get
\begin{align*}
{-\s}(\beta)
&\geq
-\frac{1}{b}\b(\beta) + {-\s}(\beta(i))
\\
&>
-\frac{1}{b}\b(\beta) + \frac{C(b-1)}{b}\b(\beta)
 \\ &=
\frac{C(b-1)-1}{b}\b(\beta).\qedhere
\end{align*}
\end{proof}
\section{Signature of positive $4$-braids}\label{sec:4braids}
In this section we provide a 
$\frac{1}{3}$-linear bound for the signature of positive $4$-braids and we prove the Main Proposition.

\begin{prop}\label{prop:thirdsigbound}
 For all non-trivial positive $4$-braids $\beta$, we have ${-\s}(\beta)>\frac{1}{3}\b(\beta)$.
\end{prop}
For $\b> 21$, Proposition~\ref{prop:thirdsigbound} follows from Corollary~\ref{cor:5/12}.
We provide a complete proof that is independent of Corollary~\ref{cor:5/12}. 
\begin{proof}[Proof of Proposition~\ref{prop:thirdsigbound}]
Let $\beta$ be a positive $4$-braid and choose a positive braid word $w= a_{i_1} a_{i_2}\cdots a_{i_l}$ for $\beta$ or cyclic permutations of $\beta$
such that the number of $ a_2$ in this braid word is minimal among all possible such positive braid words. 
Without loss of generality, we assume that $w$ does contain all three generators $a_i$ at least once.
By applying cyclic permutations and braid relations, we can arrange that the first two letters of $w$ are not $a_2$.
For example, we consider $ a_1 a_1 a_3 a_2 a_2=
~\xygraph{
!{0;/r0.7pc/:}
[u(1.5)]!{\xcapv[3]@(0)}
[lu]!{\xcapv[3]@(0)}
[lu]!{\xcapv[3]@(0)}
[lu]!{\xcapv[3]@(0)}
[u(0.5)r] !{\xcaph[1]@(0)}
[d(0.5)l]!{\xcaph[1]@(0)}
[d(0.5)] !{\xcaph[1]@(0)}
[d(0.5)lll] !{\xcaph[1]@(0)}
[d(0.5)l] !{\xcaph[1]@(0)}
}\hspace{-1.75pc}$
instead of $ a_2 a_1 a_1 a_3 a_2=
~\xygraph{
!{0;/r0.7pc/:}
[u(1.5)]!{\xcapv[3]@(0)}
[lu]!{\xcapv[3]@(0)}
[lu]!{\xcapv[3]@(0)}
[lu]!{\xcapv[3]@(0)}
[u(0.5)r] !{\xcaph[1]@(0)}
[d(0.5)] !{\xcaph[1]@(0)}
[d(0.5)lll] !{\xcaph[1]@(0)}
[d(0.5)l] !{\xcaph[1]@(0)}
[d(0.5)] !{\xcaph[1]@(0)}
}\hspace{-1.75pc}$.
Let $B_1,\ldots,B_n$ be the blocks of consecutive $a_2$ and $k_i$ the number of $ a_i$ in $w$.
Of course $k_2\geq n$ holds, and, by the assumption of minimality of the number of $a_2$ in $w$, we have at least two generators between two consecutive $B_i$, which yields $k_1+k_3\geq 2n$. Therefore,
\begin{equation}\label{eq:Skigeq3n}
\frac{k_1+k_2+k_3}{3} \geq n.
\end{equation}

We first show ${-\s}(\beta)\geq\frac{1}{3}\b(\beta)$.
Let $\beta'$ denote the braid obtained from $\beta$ by removing $B_2,B_3,\ldots,$ and $B_n$. By Remark~\ref{rem:SigIsaQM}, we have
\begin{align*}
{-\s}(\beta) &\geq  {-\s}(\beta') +\sum_{i=2}^{k}{-\s}(B_i)-(n-1) \\
&= {-\s}(\beta') +\sum_{i=2}^{k}(\l(B_i)-1)-(n-1).\end{align*}
The closure of $\beta'$ is a connected sum of the torus links \[T(2,k_1), T(2,\l(B_1))\et T(2,k_3).\] This yields ${-\s}(\beta')=k_1-1+k_3-1+\l(B_1)-1$ by Lemma~\ref{lemma:sigandunion}. Therefore,
\begin{align*}{-\s}(\beta)
&\geq k_1-1+k_3-1+\l(B_1)-1+ \sum_{i=2}^{k}(\l(B_i)-1)-(n-1)
\\
&=k_1-1+k_3-1+\sum_{i=1}^{k}(\l(B_i)-1)-(n-1)
\\
&=
k_1-1+k_3-1+k_2-n-(n-1)\\
&=
k_1+k_2+k_3-2n-1
 \geq
 \frac{k_1+k_2+k_3}{3}-1
 = \frac{\b(\beta)}{3},
\end{align*}
where in the last line \eqref{eq:Skigeq3n} and $\b(\beta)=k_1+k_2+k_3-3$ are used.

We observe that if inequality~\eqref{eq:Skigeq3n} is a strict inequality, then the above calculation proves ${-\s}(\beta)>\frac{1}{3}\b(\beta)$.
Thus, it remains to consider $w$ satisfying $\frac{k_1+k_2+k_3}{3}=n$, which implies that the inequalities
$k_2\geq n$ and $k_1+k_3\geq2n$ are equalities. Therefore, the blocks $B_i$ consist of a single $ a_2$ and in $w$ we have exactly two generators between two consecutive $B_i$.
We write $w$ as $ a_{i_1} a_{j_1} a_2 a_{i_2} a_{j_2}a_2\cdots a_{i_n} a_{j_n} a_2$, for some $i_l,j_l$ in $\{1,3\}$.

Since $\beta$ contains all types of generators and is non-trivial, we have $k_2=n\geq 2$.
Removing all but the last two $ a_2$ in $w$ yields a positive $4$-braid $\beta''$ with \[\l(\beta'') = \l(\beta)-(n-2) = 3n-(n-2) = 2n+2.\]
The braid $\beta''$ satisfies ${-\s}=\b$, which is seen as follows.
The braid $\beta''$ equals $ a_1^{i} a_3^{j} a_2\gamma a_2$ with $i+j=2n-2$, where $\gamma$ is $a_1 a_3$, $a_1^2$ or $ a_3^2$.
The closure of $ a_1^{i} a_3^{j} a_2 a_1 a_3 a_2$ with $i+j=2n-2$ is the torus link $T(2,2n)$, for which ${-\s}=\b$ holds. 
The closure of $ a_1^{i} a_3^{j} a_2 a_1 a_1 a_2$ is a connected sum of the torus link $T(2,j)$ and the closure of the $3$-braid
$ a_1^{i} a_2 a_1 a_1 a_2$
, which both satisfy ${-\s}=\b$. Similarly, the closure of $a_1^{i} a_3^{j} a_2 a_3 a_3 a_2$ is a connected sum of the torus link $T(2,i)$ and the closure of the $3$-braid
$ a_1^{j} a_2 a_1 a_1 a_2$.

Using Lemma~\ref{lemma:bandcutting}, ${-\s}(\beta'')=\b(\beta'')$ and 
$\b(\beta)=k_1+k_2+k_3-3$ we calculate
\begin{align*}
{-\s}(\beta)
&\geq
{-\s}(\beta'')-(n-2)
=
\b(\beta'')-n+2
=
2n-1-n+2
\\
&=
n+1
=
\frac{k_1+k_2+k_3}{3}+1
=
\frac{\b(\beta)}{3}+2.
\qedhere
\end{align*}
\end{proof}
The strategy for the proof of the Main Proposition is the following. To a braid $\beta$ we add roughly $\frac{1}{2}\l(\beta)$ generators
such that the resulting braid $\tilde{\beta}$ is simple enough that one can prove
${-\S}(\tilde{\beta})\geq \frac{2}{3}\l(\tilde{\beta})=\l(\beta)$. Since ${-\S}(\beta)\geq -\frac{1}{2}\l(\beta) -\S(\tilde{\beta})$ holds by Lemma~\ref{lemma:bandcutting},
we conclude that ${-\S}(\beta)\geq \frac{1}{2}\l(\beta)$ holds.
In fact, this only works for a part of the braid (at least for $\frac{2}{3}$ of the braid in terms of length) and for the rest of the braid we are only able to prove ${-\S}\geq\frac{1}{4}\l$.
Combining this yields
${-\S}(\beta)\geq\frac{2}{3}\frac{1}{2}\l(\beta)+\frac{1}{3}\frac{1}{4}\l(\beta)=\frac{5}{12}\l(\beta)$.

The braid $\tilde{\beta}$ will be obtained from $\beta$ using the following Lemma.
\begin{lemma}\label{lemma:4->6}
 Let $B$ be a positive $4$-braid of length $4$. If $B$ is not
\[a_{2} a_{1} a_{1} a_{2}=~\xygraph{
!{0;/r0.7pc/:}
[u(1.25)]!{\xcapv[2.5]@(0)}
[lu]!{\xcapv[2.5]@(0)}
[lu]!{\xcapv[2.5]@(0)}
[lu]!{\xcapv[2.5]@(0)}
[u(0.5)r] !{\xcaph[1]@(0)}
[d(0.5)ll] !{\xcaph[1]@(0)}
[d(0.5)l] !{\xcaph[1]@(0)}
[d(0.5)] !{\xcaph[1]@(0)}
}\hspace{-1.5pc} \quad\text{or}\quad a_{2} a_{3} a_{3} a_{2}=~\xygraph{
!{0;/r0.7pc/:}
[u(1.25)]!{\xcapv[2.5]@(0)}
[lu]!{\xcapv[2.5]@(0)}
[lu]!{\xcapv[2.5]@(0)}
[lu]!{\xcapv[2.5]@(0)}
[u(0.5)r] !{\xcaph[1]@(0)}
[d(0.5)] !{\xcaph[1]@(0)}
[d(0.5)l] !{\xcaph[1]@(0)}
[d(0.5)ll] !{\xcaph[1]@(0)}
}\hspace{-1pc},\]
then 
one can add two generators to $B$ such that it becomes \[\Delta= a_{1} a_{3} a_{2} a_{1} a_{3} a_{2}=~\xygraph{
!{0;/r0.7pc/:}
[u(1.25)]!{\xcapv[2.5]@(0)}
[lu]!{\xcapv[2.5]@(0)}
[lu]!{\xcapv[2.5]@(0)}
[lu]!{\xcapv[2.5]@(0)}
[u(0.5)r] !{\xcaph[1]@(0)}
[d(0.5)] !{\xcaph[1]@(0)}
[lll] !{\xcaph[1]@(0)}
[d(0.5)] !{\xcaph[1]@(0)}
[d(0.5)] !{\xcaph[1]@(0)}
[lll] !{\xcaph[1]@(0)}
}\hspace{-1.5pc},\]
\[L= a_{1} a_{2} a_{3} a_{1} a_{2} a_{3}=~\xygraph{
!{0;/r0.7pc/:}
[u(1.5)]!{\xcapv[3]@(0)}
[lu]!{\xcapv[3]@(0)}
[lu]!{\xcapv[3]@(0)}
[lu]!{\xcapv[3]@(0)}
[u(0.5)rr] !{\xcaph[1]@(0)}
[d(0.5)ll] !{\xcaph[1]@(0)}
[d(0.5)] !{\xcaph[1]@(0)}
[lll] !{\xcaph[1]@(0)}
[d(0.5)] !{\xcaph[1]@(0)}
[d(0.5)ll] !{\xcaph[1]@(0)}
}\hspace{-1.5pc}, \quad\text{or}\quad
R= a_{3} a_{2} a_{1} a_{3} a_{2} a_{1}=~\xygraph{
!{0;/r0.7pc/:}
[u(1.5)]!{\xcapv[3]@(0)}
[lu]!{\xcapv[3]@(0)}
[lu]!{\xcapv[3]@(0)}
[lu]!{\xcapv[3]@(0)}
[u(0.5)] !{\xcaph[1]@(0)}
[d(0.5)] !{\xcaph[1]@(0)}
[d(0.5)] !{\xcaph[1]@(0)}
[lll] !{\xcaph[1]@(0)}
[d(0.5)] !{\xcaph[1]@(0)}
[d(0.5)] !{\xcaph[1]@(0)}
}\hspace{-1.5pc}.\]

\end{lemma}
Here, `adding a generator to a positive braid $\beta$' means choosing some positive braid word for $\beta$ and then adding a generator $ a_i$ somewhere in this word.
\begin{proof}
We assume that $B$ is represented as a braid word such that the number of $ a_{2}$ in $B$ is minimal. For example,
$ a_{2} a_{1} a_{2} a_{1}=~\xygraph{
!{0;/r0.7pc/:}
[u(1.25)]!{\xcapv[2.5]@(0)}
[lu]!{\xcapv[2.5]@(0)}
[lu]!{\xcapv[2.5]@(0)}
[lu]!{\xcapv[2.5]@(0)}
[u(0.5)] !{\xcaph[1]@(0)}
[d(0.5)] !{\xcaph[1]@(0)}
[d(0.5)ll] !{\xcaph[1]@(0)}
[d(0.5)] !{\xcaph[1]@(0)}
}\hspace{-1.5pc}$ is not considered because it represents the same braid as $ a_{1} a_{2} a_{1} a_{1}=~\xygraph{
!{0;/r0.7pc/:}
[u(1.25)]!{\xcapv[2.5]@(0)}
[lu]!{\xcapv[2.5]@(0)}
[lu]!{\xcapv[2.5]@(0)}
[lu]!{\xcapv[2.5]@(0)}
[u(0.5)] !{\xcaph[1]@(0)}
[d(0.5)l] !{\xcaph[1]@(0)}
[d(0.5)] !{\xcaph[1]@(0)}
[d(0.5)ll] !{\xcaph[1]@(0)}
}\hspace{-1.5pc}$.
We group all possible $B$ according to the number of $a_{2}$ contained in $B$ and proceed case by case.
Cases are consider only up to rotations and reflections. Newly added generators are marked in {blue}.

\textbullet\quad The braids of length $4$ with no $ a_{2}$ are $~\xygraph{
!{0;/r0.7pc/:}
[u(1.5)]!{\xcapv[2.5]@(0)}
[lu]!{\xcapv[2.5]@(0)}
[lu]!{\xcapv[2.5]@(0)}
[lu]!{\xcapv[2.5]@(0)}
[u(0.5)] !{\xcaph[1]@(0)}
[d(0.5)l] !{\xcaph[1]@(0)}
[d(0.5)l] !{\xcaph[1]@(0)}
[d(0.5)l] !{\xcaph[1]@(0)}
}\hspace{-1.25pc}$, $~\xygraph{
!{0;/r0.7pc/:}
[u(1.25)]!{\xcapv[2]@(0)}
[lu]!{\xcapv[2]@(0)}
[lu]!{\xcapv[2]@(0)}
[lu]!{\xcapv[2]@(0)}
[u(0.5)] !{\xcaph[1]@(0)}
[d(0.5)l] !{\xcaph[1]@(0)}
[d(0.5)l] !{\xcaph[1]@(0)}
[ur] !{\xcaph[1]@(0)}
}\hspace{-1pc}$, and $~\xygraph{
!{0;/r0.7pc/:}
[u(1)]!{\xcapv[1.5]@(0)}
[lu]!{\xcapv[1.5]@(0)}
[lu]!{\xcapv[1.5]@(0)}
[lu]!{\xcapv[1.5]@(0)}
[u(0.5)] !{\xcaph[1]@(0)}
[d(0.5)l] !{\xcaph[1]@(0)}
[u(0.5)r] !{\xcaph[1]@(0)}
[d(0.5)l] !{\xcaph[1]@(0)}
}\hspace{-0.5pc}.$
In $~\xygraph{
!{0;/r0.7pc/:}
[u(1.5)]!{\xcapv[2.5]@(0)}
[lu]!{\xcapv[2.5]@(0)}
[lu]!{\xcapv[2.5]@(0)}
[lu]!{\xcapv[2.5]@(0)}
[u(0.5)] !{\xcaph[1]@(0)}
[d(0.5)l] !{\xcaph[1]@(0)}
[d(0.5)l] !{\xcaph[1]@(0)}
[d(0.5)l] !{\xcaph[1]@(0)}
}\hspace{-1.5pc}$ we first add one $ a_{2}$ to get
$\;\xygraph{
!{0;/r0.7pc/:}
[u(1.5)]!{\xcapv[3]@(0)}
[lu]!{\xcapv[3]@(0)}
[lu]!{\xcapv[3]@(0)}
[lu]!{\xcapv[3]@(0)}
[u(0.5)] !{\xcaph[1]@(0)}
[d(0.5)] !{\color{cyan}\xcaph[1]@(0)}
[d(0.5)ll] !{\color{black}\xcaph[1]@(0)}
[d(0.5)l] !{\xcaph[1]@(0)}
[d(0.5)l] !{\xcaph[1]@(0)}
}\hspace{-1.75pc}=~\xygraph{
!{0;/r0.7pc/:}
[u(1.5)]!{\xcapv[3]@(0)}
[lu]!{\xcapv[3]@(0)}
[lu]!{\xcapv[3]@(0)}
[lu]!{\xcapv[3]@(0)}
[u(0.5)r] !{\xcaph[1]@(0)}
[d(0.5)ll] !{\xcaph[1]@(0)}
[d(0.5)] !{\xcaph[1]@(0)}
[d(0.5)ll] !{\xcaph[1]@(0)}
[d(0.5)l] !{\xcaph[1]@(0)}
}\hspace{-1.75pc}$ and then add a $ a_{3}$
to get $\;\xygraph{
!{0;/r0.7pc/:}
[u(1.5)]!{\xcapv[3]@(0)}
[lu]!{\xcapv[3]@(0)}
[lu]!{\xcapv[3]@(0)}
[lu]!{\xcapv[3]@(0)}
[u(0.5)r] !{\xcaph[1]@(0)}
[d(0.5)] !{\color{cyan}\xcaph[1]@(0)}
[lll] !{\color{black}\xcaph[1]@(0)}
[d(0.5)] !{\xcaph[1]@(0)}
[d(0.5)ll] !{\xcaph[1]@(0)}
[d(0.5)l] !{\xcaph[1]@(0)}
}\hspace{-1.75pc}=~\xygraph{
!{0;/r0.7pc/:}
[u(1.75)]!{\xcapv[3.5]@(0)}
[lu]!{\xcapv[3.5]@(0)}
[lu]!{\xcapv[3.5]@(0)}
[lu]!{\xcapv[3.5]@(0)}
[u(0.5)r] !{\xcaph[1]@(0)}
[d(0.5)] !{\color{black}\xcaph[1]@(0)}
[d(0.5)ll] !{\color{black}\xcaph[1]@(0)}
[d(0.5)ll] !{\xcaph[1]@(0)}
[d(0.5)] !{\xcaph[1]@(0)}
[d(0.5)ll] !{\xcaph[1]@(0)}
}\hspace{-2pc}=~\xygraph{
!{0;/r0.7pc/:}
[u(1.5)]!{\xcapv[3]@(0)}
[lu]!{\xcapv[3]@(0)}
[lu]!{\xcapv[3]@(0)}
[lu]!{\xcapv[3]@(0)}
[u(0.5)rr] !{\xcaph[1]@(0)}
[d(0.5)ll] !{\xcaph[1]@(0)}
[d(0.5)] !{\xcaph[1]@(0)}
[lll] !{\xcaph[1]@(0)}
[d(0.5)] !{\xcaph[1]@(0)}
[d(0.5)ll] !{\xcaph[1]@(0)}
}\hspace{-1.5pc}=~L$.
In the other cases we add two $ a_{2}$ as follows.
\[\;\xygraph{
!{0;/r0.7pc/:}
[u(1.5)]!{\xcapv[3]@(0)}
[lu]!{\xcapv[3]@(0)}
[lu]!{\xcapv[3]@(0)}
[lu]!{\xcapv[3]@(0)}
[u(0.5)r] !{\color{cyan}\xcaph[1]@(0)}
[d(0.5)ll] !{\color{black}\xcaph[1]@(0)}
[d(0.5)] !{\color{cyan}\xcaph[1]@(0)}
[d(0.5)ll] !{\color{black}\xcaph[1]@(0)}
[d(0.5)l] !{\xcaph[1]@(0)}
[u(1.5)r] !{\xcaph[1]@(0)}
}\hspace{-1.75pc}=L \et \;\xygraph{
!{0;/r0.7pc/:}
[u(1.25)]!{\xcapv[2.5]@(0)}
[lu]!{\xcapv[2.5]@(0)}
[lu]!{\xcapv[2.5]@(0)}
[lu]!{\xcapv[2.5]@(0)}
[u(0.5)r] !{\color{cyan}\xcaph[1]@(0)}
[d(0.5)ll] !{\color{black}\xcaph[1]@(0)}
[r] !{\xcaph[1]@(0)}
[d(0.5)ll] !{\color{cyan}\xcaph[1]@(0)}
[d(0.5)ll] !{\color{black}\xcaph[1]@(0)}
[r] !{\xcaph[1]@(0)}
}\hspace{-1.5pc}=\Delta.\]

\textbullet\quad The following are all $B$ with one $ a_{2}$. We have always indicated how to add two generators ({blue}) yielding $L,R$, or $\Delta$.
\[
\xygraph{
!{0;/r0.7pc/:}
[u(1.5)]!{\xcapv[3]@(0)}
[lu]!{\xcapv[3]@(0)}
[lu]!{\xcapv[3]@(0)}
[lu]!{\xcapv[3]@(0)}
[u(0.5)r] !{\xcaph[1]@(0)}
[d(0.5)] !{\color{cyan}\xcaph[1]@(0)}
[lll] !{\color{black}\xcaph[1]@(0)}
[d(0.5)] !{\color{cyan}\xcaph[1]@(0)}
[d(0.5)ll] !{\color{black}\xcaph[1]@(0)}
[d(0.5)l] !{\xcaph[1]@(0)}
}\hspace{-1.75pc}=L, \hspace{1pc}\xygraph{
!{0;/r0.7pc/:}
[u(1.5)]!{\xcapv[3]@(0)}
[lu]!{\xcapv[3]@(0)}
[lu]!{\xcapv[3]@(0)}
[lu]!{\xcapv[3]@(0)}
[u(0.5)r] !{\xcaph[1]@(0)}
[d(0.5)] !{\xcaph[1]@(0)}
[lll] !{\xcaph[1]@(0)}
[d(0.5)] !{\color{cyan}\xcaph[1]@(0)}
[d(0.5)ll] !{\color{black}\xcaph[1]@(0)}
[d(0.5)l] !{\color{cyan}\xcaph[1]@(0)\color{black}}
}\hspace{-1.5pc}=L,\hspace{1pc}\xygraph{
!{0;/r0.7pc/:}
[u(1.5)]!{\xcapv[3]@(0)}
[lu]!{\xcapv[3]@(0)}
[lu]!{\xcapv[3]@(0)}
[lu]!{\xcapv[3]@(0)}
[u(0.5)r] !{\color{cyan}\xcaph[1]@(0)}
[d(0.5)] !{\xcaph[1]@(0)}
[lll] !{\color{black}\xcaph[1]@(0)}
[d(0.5)] !{\xcaph[1]@(0)}
[d(0.5)ll] !{\xcaph[1]@(0)}
[d(0.5)l] !{\xcaph[1]@(0)}
}\hspace{-1.75pc}=L,\et\xygraph{
!{0;/r0.7pc/:}
[u(1.5)]!{\xcapv[3]@(0)}
[lu]!{\xcapv[3]@(0)}
[lu]!{\xcapv[3]@(0)}
[lu]!{\xcapv[3]@(0)}
[u(0.5)r] !{\color{cyan}\xcaph[1]@(0)}
[d(0.5)] !{\color{black}\xcaph[1]@(0)}
[lll] !{\color{cyan}\xcaph[1]@(0)}
[d(0.5)] !{\color{black}\xcaph[1]@(0)}
[d(0.5)ll] !{\xcaph[1]@(0)}
[d(0.5)l] !{\xcaph[1]@(0)}
}\hspace{-1.75pc}=L
\]

\textbullet\quad If $B$ contains two $ a_{2}$, but is not
\[a_{2} a_{1} a_{1} a_{2}=\xygraph{
!{0;/r0.7pc/:}
[u(1.25)]!{\xcapv[2.5]@(0)}
[lu]!{\xcapv[2.5]@(0)}
[lu]!{\xcapv[2.5]@(0)}
[lu]!{\xcapv[2.5]@(0)}
[u(0.5)r] !{\xcaph[1]@(0)}
[d(0.5)ll] !{\xcaph[1]@(0)}
[d(0.5)l] !{\xcaph[1]@(0)}
[d(0.5)] !{\xcaph[1]@(0)}
}\hspace{-1.5pc}\vel a_{2} a_{3} a_{3} a_{2}=\xygraph{
!{0;/r0.7pc/:}
[u(1.25)]!{\xcapv[2.5]@(0)}
[lu]!{\xcapv[2.5]@(0)}
[lu]!{\xcapv[2.5]@(0)}
[lu]!{\xcapv[2.5]@(0)}
[u(0.5)r] !{\xcaph[1]@(0)}
[d(0.5)] !{\xcaph[1]@(0)}
[d(0.5)l] !{\xcaph[1]@(0)}
[d(0.5)ll] !{\xcaph[1]@(0)}
}\hspace{-1.5pc},\]
then it is one of the following (as before, it is indicated in {blue} which generators to add).
\[\;\xygraph{
!{0;/r0.7pc/:}
[u(1.5)]!{\xcapv[3]@(0)}
[lu]!{\xcapv[3]@(0)}
[lu]!{\xcapv[3]@(0)}
[lu]!{\xcapv[3]@(0)}
[u(0.5)r] !{\xcaph[1]@(0)}
[d(0.5)] !{\color{cyan}\xcaph[1]@(0)}
[lll] !{\xcaph[1]@(0)}
[d(0.5)] !{\color{black}\xcaph[1]@(0)}
[d(0.5)ll] !{\xcaph[1]@(0)}
[d(0.5)l] !{\xcaph[1]@(0)}
}\hspace{-1.75pc}=L,\hspace{1pc}\xygraph{
!{0;/r0.7pc/:}
[u(1.25)]!{\xcapv[2.5]@(0)}
[lu]!{\xcapv[2.5]@(0)}
[lu]!{\xcapv[2.5]@(0)}
[lu]!{\xcapv[2.5]@(0)}
[u(0.5)r] !{\xcaph[1]@(0)}
[d(0.5)] !{\color{cyan}\xcaph[1]@(0)}
[lll] !{\xcaph[1]@(0)}
[d(0.5)] !{\color{black}\xcaph[1]@(0)}
[d(0.5)ll] !{\xcaph[1]@(0)}
[r] !{\xcaph[1]@(0)}
}\hspace{-1.75pc}=\Delta,\hspace{1pc}\xygraph{
!{0;/r0.7pc/:}
[u(1.25)]!{\xcapv[2.5]@(0)}
[lu]!{\xcapv[2.5]@(0)}
[lu]!{\xcapv[2.5]@(0)}
[lu]!{\xcapv[2.5]@(0)}
[u(0.5)r] !{\xcaph[1]@(0)}
[d(0.5)ll] !{\xcaph[1]@(0)}
[d(0.5)] !{\xcaph[1]@(0)}
[d(0.5)] !{\color{cyan}\xcaph[1]@(0)}
[lll] !{\xcaph[1]@(0)}
[ur] !{\color{black}\xcaph[1]@(0)}
}\hspace{-1.5pc}=\Delta,\hspace{1pc}\xygraph{
!{0;/r0.7pc/:}
[u(1.5)]!{\xcapv[3]@(0)}
[lu]!{\xcapv[3]@(0)}
[lu]!{\xcapv[3]@(0)}
[lu]!{\xcapv[3]@(0)}
[u(0.5)] !{\xcaph[1]@(0)}
[d(0.5)] !{\xcaph[1]@(0)}
[d(0.5)] !{\color{cyan}\xcaph[1]@(0)}
[lll] !{\xcaph[1]@(0)}
[d(0.5)] !{\color{black}\xcaph[1]@(0)}
[d(0.5)ll] !{\xcaph[1]@(0)}
}\hspace{-1.75pc}=\Delta,~\xygraph{
!{0;/r0.7pc/:}
[u(1.25)]!{\xcapv[3]@(0)}
[lu]!{\xcapv[3]@(0)}
[lu]!{\xcapv[3]@(0)}
[lu]!{\xcapv[3]@(0)}
[u(0.5)rr] !{\xcaph[1]@(0)}
[d(0.5)ll] !{\xcaph[1]@(0)}
[d(0.5)] !{\color{cyan}\xcaph[1]@(0)}
[lll] !{\xcaph[1]@(0)}
[d(0.5)] !{\color{black}\xcaph[1]@(0)}
[d(0.5)ll] !{\xcaph[1]@(0)}
}\hspace{-1.75pc}=L.\]

\textbullet\quad Finally, there are the following two $B$ with three or four $ a_{2}$.
\[\xygraph{
!{0;/r0.7pc/:}
[u(2)]!{\xcapv[3.5]@(0)}
[lu]!{\xcapv[3.5]@(0)}
[lu]!{\xcapv[3.5]@(0)}
[lu]!{\xcapv[3.5]@(0)}
[u(0.5)r] !{\xcaph[1]@(0)}
[d(0.5)] !{\color{cyan}\xcaph[1]@(0)}
[d(0.5)ll] !{\color{black}\xcaph[1]@(0)}
[d(0.5)ll] !{\color{cyan}\xcaph[1]@(0)}
[d(0.5)] !{\color{black}\xcaph[1]@(0)}
[d(0.5)ll] !{\xcaph[1]@(0)}
}\hspace{-1.75pc}=L\et\xygraph{
!{0;/r0.7pc/:}
[u(2)]!{\xcapv[3.5]@(0)}
[lu]!{\xcapv[3.5]@(0)}
[lu]!{\xcapv[3.5]@(0)}
[lu]!{\xcapv[3.5]@(0)}
[u(0.5)r] !{\xcaph[1]@(0)}
[d(0.5)] !{\color{cyan}\xcaph[1]@(0)}
[d(0.5)ll] !{\color{black}\xcaph[1]@(0)}
[d(0.5)l] !{\xcaph[1]@(0)}
[d(0.5)ll] !{\color{cyan}\xcaph[1]@(0)}
[d(0.5)] !{\color{black}\xcaph[1]@(0)}
}\hspace{-1.75pc}=L.\qedhere\]
\end{proof}

\begin{proof}[Proof of the Main Proposition]
Let $\beta= a_{i_1} a_{i_2}\cdots a_{i_{\l}}$ be a positive 4-braid of length $\l$. We fix a positive integer $n$ that is a multiple of $4$ and study $\beta^n$, which is a braid of length $n\l$.
First, we write $\beta^n$ as $B_1B_2\cdots B_{\frac{n\l}{4}}$, where every $B_i$ is a positive braid of length $4$.
Let $k$ be the number of $ a_{2} a_{1} a_{1} a_{2}=~\xygraph{
!{0;/r0.7pc/:}
[u(1.25)]!{\xcapv[2.5]@(0)}
[lu]!{\xcapv[2.5]@(0)}
[lu]!{\xcapv[2.5]@(0)}
[lu]!{\xcapv[2.5]@(0)}
[u(0.5)r] !{\xcaph[1]@(0)}
[d(0.5)ll] !{\xcaph[1]@(0)}
[d(0.5)l] !{\xcaph[1]@(0)}
[d(0.5)] !{\xcaph[1]@(0)}
}\hspace{-1.5pc}$ and $ a_{2} a_{3} a_{3} a_{2}=~\xygraph{
!{0;/r0.7pc/:}
[u(1.25)]!{\xcapv[2.5]@(0)}
[lu]!{\xcapv[2.5]@(0)}
[lu]!{\xcapv[2.5]@(0)}
[lu]!{\xcapv[2.5]@(0)}
[u(0.5)r] !{\xcaph[1]@(0)}
[d(0.5)] !{\xcaph[1]@(0)}
[d(0.5)l] !{\xcaph[1]@(0)}
[d(0.5)ll] !{\xcaph[1]@(0)}
}\hspace{-1.5pc}$ among the $B_i$.

We may assume that $k$ is less or equal than $\frac{1}{3}\frac{n\l}{4}=\frac{n\l}{12}$.
For if this were not the case, we switch $\beta^n$ to one of the cyclic permutations $\beta^n_1= a_{i_1}^{-1}\beta^n a_{i_1}$ or $\beta^n_2= a_{i_2}^{-1} a_{i_1}^{-1}\beta^n a_{i_1} a_{i_2}$,
which have the same closure 
as $\beta^n$.
It is easy to see that if we decompose $\beta^n,\beta^n_1,$ and $\beta^n_2$ into blocks of length $4$ and add up the number of
$~\xygraph{
!{0;/r0.7pc/:}
[u(1.25)]!{\xcapv[2.5]@(0)}
[lu]!{\xcapv[2.5]@(0)}
[lu]!{\xcapv[2.5]@(0)}
[lu]!{\xcapv[2.5]@(0)}
[u(0.5)r] !{\xcaph[1]@(0)}
[d(0.5)ll] !{\xcaph[1]@(0)}
[d(0.5)l] !{\xcaph[1]@(0)}
[d(0.5)] !{\xcaph[1]@(0)}
}\hspace{-1.5pc}$ and $~\xygraph{
!{0;/r0.7pc/:}
[u(1.25)]!{\xcapv[2.5]@(0)}
[lu]!{\xcapv[2.5]@(0)}
[lu]!{\xcapv[2.5]@(0)}
[lu]!{\xcapv[2.5]@(0)}
[u(0.5)r] !{\xcaph[1]@(0)}
[d(0.5)] !{\xcaph[1]@(0)}
[d(0.5)l] !{\xcaph[1]@(0)}
[d(0.5)ll] !{\xcaph[1]@(0)}
}\hspace{-1.5pc}$ in all three decompositions,
we get at most $\frac{n \l}{4}$; thus, $k\leq \frac{1}{3}\frac{n\l}{4}$ for at least one of $\beta^n,\beta^n_1,$ and $\beta^n_2$.

Now, we apply Lemma~\ref{lemma:4->6} to change $\beta^n$ to $\widetilde{\beta^n}=\widetilde{B_1}\widetilde{B_2}\cdots \widetilde{B_{\frac{n\l}{4}}}$,
where the $\widetilde{B_i}$ are braid words of length 4 or 6 that are chosen as follows. If $B_i$ is
$~\xygraph{
!{0;/r0.7pc/:}
[u(1.25)]!{\xcapv[2.5]@(0)}
[lu]!{\xcapv[2.5]@(0)}
[lu]!{\xcapv[2.5]@(0)}
[lu]!{\xcapv[2.5]@(0)}
[u(0.5)r] !{\xcaph[1]@(0)}
[d(0.5)ll] !{\xcaph[1]@(0)}
[d(0.5)l] !{\xcaph[1]@(0)}
[d(0.5)] !{\xcaph[1]@(0)}
}\hspace{-1.5pc}$ or $~\xygraph{
!{0;/r0.7pc/:}
[u(1.25)]!{\xcapv[2.5]@(0)}
[lu]!{\xcapv[2.5]@(0)}
[lu]!{\xcapv[2.5]@(0)}
[lu]!{\xcapv[2.5]@(0)}
[u(0.5)r] !{\xcaph[1]@(0)}
[d(0.5)] !{\xcaph[1]@(0)}
[d(0.5)l] !{\xcaph[1]@(0)}
[d(0.5)ll] !{\xcaph[1]@(0)}
}\hspace{-1.5pc}$, then $\widetilde{B_i}$ is $B_i$.
Otherwise, $\widetilde{B_i}$ is equal to $L,R,$ or $\Delta$ such that $\widetilde{B_i}$ can be obtained from $B_i$ by adding $2$ generators, which is possible by Lemma~\ref{lemma:4->6}.
By Lemma~\ref{lemma:bandcutting} we have \[{-\s}(\beta^n) \geq  {-\s}(\widetilde{\beta^n})-2(\frac{nl}{4}-k).\]

For a given braid $\alpha$, let $\alpha^\text{rot}$ denote the braid obtained from a braid diagram for $\alpha$ by a planar rotation of $180$ degrees. Then the following holds
\begin{equation}\label{eq:sigbetabetar}
{-\s}(\widetilde{\beta^n}(\widetilde{\beta^n})^\text{rot}) \geq  2k {-\s}(\Delta^{2(\frac{n\l}{4}-k)}) = 2k+8(\frac{n\l}{4}-k)-1.
\end{equation}
Before proving \eqref{eq:sigbetabetar}, we use it to finish the proof.
Since $\widetilde{\beta^n}$ and $(\widetilde{\beta^n})^\text{rot}$ have the same closure (up to changing the orientation) and $\s$ is a quasi-morphism of defect $3$,
we get \begin{align*}
-2\s(\widetilde{\beta^n}) &= {-\s}(\widetilde{\beta^n}){-\s}(\widetilde{\beta^n}^\text{rot}) \geq {-\s}(\widetilde{\beta^n}(\widetilde{\beta^n})^\text{rot})-3\\
&\overset{\eqref{eq:sigbetabetar}}{\geq}\left(2k+8(\tfrac{n\l}{4}-k)-1\right)-3 = -6k+2n\l-4.\end{align*}
Therefore,
\begin{align*}
{-\s}(\beta^n)
&\geq
{-\s}(\widetilde{\beta^n})-2(\frac{n\l}{4}-k)
 \geq
-3k+n\l-2-2(\frac{n\l}{4}-k)
\\
&=
-k+\frac{n\l}{2}-2
 \geq
-\frac{n\l}{12}+\frac{nl}{2}-2
 = \frac{5n\l}{12}-2,
\end{align*}
and thus \[
{-\S}(\beta) = \lim_{n\to\infty}\frac{\s(\beta^n)}{n} \geq  \frac{5\l}{12} = \frac{5}{12}\l(\beta).
         \]

It remains to prove~\eqref{eq:sigbetabetar}. For this we use that the full twist on $4$ strands \[LL = RR = \Delta\Delta = \Delta^2\] 
commutes with every $4$-braid:
for all $\alpha$ in $B_4$, we have $\alpha\Delta^2=\Delta^2\alpha$; compare~\cite{Garside_69_TheBraidGroup}.

Let us study $\widetilde{\beta^n}(\widetilde{\beta^n})^\text{rot}=\widetilde{B_1}\widetilde{B_2}\cdots \widetilde{B_{\frac{n\l}{4}}}\widetilde{B_{\frac{n\l}{4}}}^\text{rot}\cdots \widetilde{B_1}^\text{rot}.$
The braid diagrams for $L,R$, and $\Delta$ and their rotation by $180$ degrees define the same braid: we have \[L=~\xygraph{
!{0;/r0.7pc/:}
[u(1.5)]!{\xcapv[3]@(0)}
[lu]!{\xcapv[3]@(0)}
[lu]!{\xcapv[3]@(0)}
[lu]!{\xcapv[3]@(0)}
[u(0.5)rr] !{\xcaph[1]@(0)}
[d(0.5)ll] !{\xcaph[1]@(0)}
[d(0.5)] !{\xcaph[1]@(0)}
[lll] !{\xcaph[1]@(0)}
[d(0.5)] !{\xcaph[1]@(0)}
[d(0.5)ll] !{\xcaph[1]@(0)}
}\hspace{-1.75pc}=L^\text{rot},
~R
=~\xygraph{
!{0;/r0.7pc/:}
[u(1.5)]!{\xcapv[3]@(0)}
[lu]!{\xcapv[3]@(0)}
[lu]!{\xcapv[3]@(0)}
[lu]!{\xcapv[3]@(0)}
[u(0.5)] !{\xcaph[1]@(0)}
[d(0.5)] !{\xcaph[1]@(0)}
[d(0.5)] !{\xcaph[1]@(0)}
[lll] !{\xcaph[1]@(0)}
[d(0.5)] !{\xcaph[1]@(0)}
[d(0.5)] !{\xcaph[1]@(0)}
}\hspace{-1.75pc}
=R^\text{rot}
\text{ and }\Delta
=~\xygraph{
!{0;/r0.7pc/:}
[u(1.25)]!{\xcapv[2.5]@(0)}
[lu]!{\xcapv[2.5]@(0)}
[lu]!{\xcapv[2.5]@(0)}
[lu]!{\xcapv[2.5]@(0)}
[u(0.5)r] !{\xcaph[1]@(0)}
[d(0.5)] !{\xcaph[1]@(0)}
[lll] !{\xcaph[1]@(0)}
[d(0.5)] !{\xcaph[1]@(0)}
[d(0.5)] !{\xcaph[1]@(0)}
[lll] !{\xcaph[1]@(0)}
}\hspace{-1.5pc}
=~\xygraph{
!{0;/r0.7pc/:}
[u(1.25)]!{\xcapv[2.5]@(0)}
[lu]!{\xcapv[2.5]@(0)}
[lu]!{\xcapv[2.5]@(0)}
[lu]!{\xcapv[2.5]@(0)}
[u(0.5)rr] !{\xcaph[1]@(0)}
[lll] !{\xcaph[1]@(0)}
[d(0.5)] !{\xcaph[1]@(0)}
[d(0.5)] !{\xcaph[1]@(0)}
[lll] !{\xcaph[1]@(0)}
[d(0.5)] !{\xcaph[1]@(0)}
}\hspace{-1.75pc}
=\Delta^\text{rot}.\] Therefore, if
$\widetilde{B_{\frac{n\l}{4}}}$ is $L,R$, or $\Delta$, then
$\widetilde{B_{\frac{n\l}{4}}}(\widetilde{B_{\frac{n\l}{4}}})^\text{rot}=\Delta^2$; and thus,
\begin{align*}
\widetilde{\beta^n}(\widetilde{\beta^n})^\text{rot}
&=
\widetilde{B_1}\widetilde{B_2}\cdots \widetilde{B_{\frac{n\l}{4}-1}}\Delta^2\widetilde{B_{\frac{n\l}{4}-1}}^\text{rot}\cdots \widetilde{B_1}^\text{rot}
\\
&=
\Delta^2\widetilde{B_1}\widetilde{B_2}\cdots \widetilde{B_{\frac{n\l}{4}-1}}\widetilde{B_{\frac{n\l}{4}-1}}^\text{rot}\cdots \widetilde{B_1}^\text{rot}.
\end{align*}
Otherwise, that is if $\widetilde{B_{\frac{n\l}{4}}}$ is
$~\xygraph{
!{0;/r0.7pc/:}
[u(1.25)]!{\xcapv[2.5]@(0)}
[lu]!{\xcapv[2.5]@(0)}
[lu]!{\xcapv[2.5]@(0)}
[lu]!{\xcapv[2.5]@(0)}
[u(0.5)r] !{\xcaph[1]@(0)}
[d(0.5)ll] !{\xcaph[1]@(0)}
[d(0.5)l] !{\xcaph[1]@(0)}
[d(0.5)] !{\xcaph[1]@(0)}
}\hspace{-1.5pc}$ or $~\xygraph{
!{0;/r0.7pc/:}
[u(1.25)]!{\xcapv[2.5]@(0)}
[lu]!{\xcapv[2.5]@(0)}
[lu]!{\xcapv[2.5]@(0)}
[lu]!{\xcapv[2.5]@(0)}
[u(0.5)r] !{\xcaph[1]@(0)}
[d(0.5)] !{\xcaph[1]@(0)}
[d(0.5)l] !{\xcaph[1]@(0)}
[d(0.5)ll] !{\xcaph[1]@(0)}
}\hspace{-1.5pc}$, we apply Corollary~\ref{cor:SigIsaQM} to get
\[{-\s}(\widetilde{\beta^n}(\widetilde{\beta^n})^\text{rot}) \geq  -3 {-\s}(\widetilde{B_{\frac{n\l}{4}}}\widetilde{B_{\frac{n\l}{4}}}^\text{rot})
{-\s}(\widetilde{B_1}\widetilde{B_2}\cdots \widetilde{B_{\frac{n\l}{4}-1}}\widetilde{B_{\frac{n\l}{4}-1}}^\text{rot}\cdots \widetilde{B_1}^\text{rot}),\]
where \[{-\s}(\widetilde{B_{\frac{n\l}{4}}}\widetilde{B_{\frac{n\l}{4}}}^\text{rot})=\b(\widetilde{B_{\frac{n\l}{4}}}\widetilde{B_{\frac{n\l}{4}}}^\text{rot})=5\]
since the closure of $\widetilde{B_{\frac{n\l}{4}}}(\widetilde{B_{\frac{n\l}{4}}})^\text{rot}$ is a connected sum of two $T(2,2)$ and one $T(2,4)$.
Therefore, we have 
\[{-\s}(\widetilde{\beta^n}(\widetilde{\beta^n})^\text{rot}) \geq  2 {-\s}(\widetilde{B_1}\widetilde{B_2}\cdots \widetilde{B_{\frac{n\l}{4}-1}}\widetilde{B_{\frac{n\l}{4}-1}}^\text{rot}\cdots \widetilde{B_1}^\text{rot}).\]

Applying the same argument to
\[\widetilde{B_1}\widetilde{B_2}\cdots \widetilde{B_{\frac{n\l}{4}-1}}\widetilde{B_{\frac{n\l}{4}-1}}^\text{rot}\cdots \widetilde{B_1}^\text{rot}\vel\Delta^2\widetilde{B_1}\widetilde{B_2}\cdots \widetilde{B_{\frac{n\l}{4}-1}}\widetilde{B_{\frac{n\l}{4}-1}}^\text{rot}\cdots \widetilde{B_1}^\text{rot},\]
respectively, and
continuing inductively, we get
\[{-\s}(\widetilde{\beta^n}(\widetilde{\beta^n})^\text{rot}) \geq  2k {-\s}(\Delta^{2(\frac{n\l}{4}-k)}).\]

Now~\eqref{eq:sigbetabetar} follows from 
Murasugi's formula for the signature of torus links of braid index $4$, which implies that $\s(\Delta^{2j})=\s(T(4,4j))=-8j+1$ holds for all positive integers $j$; see~\cite[Proposition 9.2]{Murasugi_OnClosed3Braids} or use~\cite[Theorem~5.2]{GLM}.
\end{proof}

\def\cprime{$'$} \def\cprime{$'$}

\end{document}